\documentclass[11 pt, article]{IEEEtran}

  \onecolumn
\usepackage{amsmath, amssymb,amsthm,cite}
\usepackage[dvips]{graphics}
\usepackage[utf8x]{inputenc}
\usepackage{graphicx}
\usepackage{psfrag}
\usepackage{bbm}
\usepackage{subfigure}
\usepackage{dblfloatfix}
\usepackage{mathtools}
\usepackage{algpseudocode}
\usepackage{algorithm}
\usepackage{hyperref}
\usepackage{color}
\usepackage{tcolorbox}
\usepackage{etoolbox}
\usepackage{mleftright}

\DeclareMathOperator*{\nn}{\nonumber}

\DeclareMathOperator{\E}{\mathbb{E}}
\allowdisplaybreaks
\newcommand{\RNum}[1]{\uppercase\expandafter{\romannumeral #1\relax}}
\newtheorem{lemma}{Lemma}
\newtheorem{theorem}{Theorem}
\newtheorem{corollary}{Corollary}

\newtheorem{remark}{Remark}
\theoremstyle{definition}

\makeatletter
\def\blfootnote{\gdef\@thefnmark{}\@foot notetext}
\makeatother

\renewcommand{\v}[1]{\mathbf{#1}}
\DeclareMathOperator{\vx}{\mathbf{x}}
\DeclareMathOperator{\vf}{\mathbf{f}}
\DeclareMathOperator{\vu}{\mathbf{u}}
\DeclareMathOperator{\vv}{\mathbf{v}}
\DeclareMathOperator{\vy}{\mathbf{y}}

\DeclareMathOperator{\PG}{\mathbb{G}}
\DeclareMathOperator{\snr}{\text{\normalfont{SNR}}}
\DeclareMathOperator{\tr}{\mathrm{T}}
\DeclareMathOperator{\Trace}{\operatorname{Tr}}

\DeclarePairedDelimiterX{\norm}[1]{\lVert}{\rVert}{#1}

\title{Reducing the LQG Cost with Minimal Communication}

 \author{Oron Sabag, Peida Tian, Victoria Kostina and Babak Hassibi
\thanks{A preliminary version of this work has been published in \cite{sabag_cdc}. The authors are with California Institute of Technology (e-mails:
 \{oron,ptian,vkostina,hassibi\}@caltech.edu).}
}
\begin{document}

\maketitle
\thispagestyle{empty}
\pagestyle{empty}

\begin{abstract}
We study the linear quadratic Gaussian (LQG) control problem, in which the controller's observation of the system state is such that a desired cost is unattainable. To achieve the desired LQG cost, we introduce a communication link from the observer (encoder) to the controller. We investigate the optimal trade-off between the improved LQG cost and the consumed communication (information) resources, measured with the conditional directed information, across all encoding-decoding policies. The main result is a semidefinite programming formulation for that optimization problem in the finite-horizon scenario, which applies to time-varying linear dynamical systems. This result extends a seminal work by Tanaka et al., where the only information the controller knows about the system state arrives via a communication channel, to the scenario where the controller has also access to a noisy observation of the system state. As part of our derivation to show the optimiality of an encoder that transmits a memoryless Gaussian measurement of the state, we show that the presence of the controller's observations at the encoder can not reduce the minimal directed information. For time-invariant systems, where the optimal policy may be time-varying, we show in the infinite-horizon scenario that the optimal policy is time-invariant and can be computed explicitly from a solution of a finite-dimensional semidefinite programming. The results are  demonstrated via examples that show that even low-quality measurements can have a significant impact on the required communication resources.
\end{abstract}


\section{Introduction}

Networked control systems share an inherent tension between the control performance and the resources that are allocated to communicate by different nodes of the system. Despite the great advances on important questions in this theme such as data rate theorems for stabilizability of dynamical systems \cite{ConCom_Elia,ConCom_matveev,ConCom_matveev2,ConCom_Liberzon2,YukselStability,TaikondaSahaiMitter,KostinaBitsArxiv,sabag_isit_fixedrate}, there are still fundamental questions that remain open such as the trade-off between communication resources and the control cost \cite{silva_framework,Milan_UB_causalRDF_concom,victoria_hassibi_tradeoff,anatoly_fixedrate,charalambos_nonanticipative,STAVROU_SI,tishby_LQG}. In this paper, we investigate this question on a simple topology consisting of the classical Linear Quadratic Gaussian (LQG) setting with a single communication link.

\begin{figure}[t]
\centering
    \psfrag{X}[][][1]{$\vx_t$}
    \psfrag{Z}[][][1]{$\vy_t$}
    \psfrag{U}[][][1]{$\vu_t$}
    \psfrag{S}[][][1]{$\vx_{t+1} = A_t\vx_t +B_t \vu_t + \v{w}_t$}
    \psfrag{F}[][][1]{(Encoder)}
    \psfrag{E}[][][1]{Observer}
    \psfrag{C}[][][1]{Controller}
    \psfrag{D}[][][1]{(Decoder)}
    \psfrag{O}[][][1]{Measurement}
    \psfrag{P}[][][.8]{$\vy_{t} = C_t\vx_t+ \vv_t$}
    \psfrag{R}[][][1]{Communication}
    \psfrag{T}[][][1]{The dynamical system}
    \includegraphics[scale = 0.65]{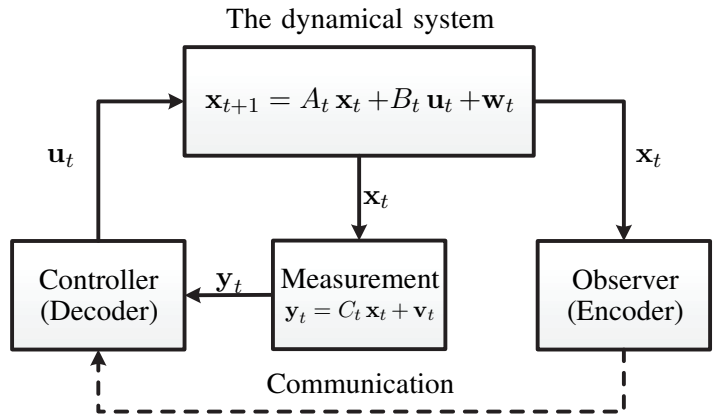}
    \caption{The LQG setting with a noisy observation $\vy_t$. The control performance (the quadratic cost) is improved using a communication link (the dashed line) from the observer to the controller.}
    \label{fig:setting}
\end{figure}


The networked control setting investigated in this paper (Fig. \ref{fig:setting}) aims to reduce the achievable control cost at the expense of communication resources. The communication link introduced between an encoder and a decoder (co-located with the controller) serves as an information pipeline to the controller that also has an access to the LQG measurements $\vy_t$. Based on its (full) observation of the state, the encoder transmits extra information to the controller resulting in a reduction in the LQG cost. One can also view this setting as the standard rate-constrained LQG setting \cite{Tanaka_SDP}, but with side information available to the controller (the measurement $\vy_t$) \cite{Kostina_SI_Allerton,STAVROU_SI,Tanaka_SI,lev2020schemes}. The objective of this paper is to characterize the minimal communication resources subject to a strict constraint on the control performance measured by a quadratic cost.

The communication (information) resources are measured with the conditional directed information. The directed information is suitable for scenarios where the operations of the involved units are sequential, e.g., channels with feedback in communication \cite{PermuterWeissmanGoldsmith09,TatikondaMitter_IT09,Kramer98} and the causal rate distortion function in the context of control problems \cite{Milan_UB_causalRDF_concom,charalambos_nonanticipative}. Also here, both mappings of the encoder and the controller are sequential and the directed information serves as a lower bound to the operational variable-length (prefix) coding problem  \cite{tanaka_scheme,Milan_UB_causalRDF_concom} (See also Section \ref{sec:conclusions}). The control performance is measured by a quadratic cost function of the state and control signals. The optimization problem is formulated for two scenarios corresponding to the finite-horizon and infinite-horizon regimes.

For the finite-horizon problem, time-varying linear dynamical systems are investigated and the minimal conditional directed information is formulated as a convex optimization problem. The optimization problem has a semidefinite programming (SDP) form (more precisely, {\tt max log-det} form) and can be implemented using standard solvers even for {large} horizons. We also show that the solution to the optimization problem can be realized by three design steps: controller gains computation, solution for the convex optimization problem and a standard Kalman filter. For the infinite-horizon problem where the dynamical system matrices are time-invariant, we show that the optimization problem can be also formulated as an SDP with the optimization variables being two positive semidefinite matrices of finite dimensions. Most importantly, we show that the optimal encoding policy is a simple, time-invariant Gaussian measurement of the state that can be computed from the convex optimization.

Our results generalize the work by Tanaka et al. \cite{Tanaka_SDP}, which introduced the SDP approach for solving control-communication problems \cite{Tanaka_Gaussian_AC}. Specifically, we investigate the full LQG setting, while \cite{Tanaka_SDP} assumed that the LQG measurement is absent ($\vy_t=0$ in Fig. \ref{fig:setting}). Thus, the control performance in our setting relies on the fusion of both the communication link information and the LQG Gaussian measurement.

Two key changes in the SDP formulation are the objective function that includes a new term due to the study of conditional directed information rather than the directed information in \cite{Tanaka_SDP}, and a new linear matrix inequality (LMI) constraint which represents the error covariance reduction due to the LQG measurement. To find the optimal policy structure, we study a relaxed optimization problem where the LQG measurements are available to the encoder as well. We then show that even in this relaxed scenario, the optimal encoder signaling is a memoryless Gaussian measurement of the state. Thus, the knowledge of the LQG measurements at the encoder can not reduce the minimal communication resources. This extends the observation made in \cite{Kostina_SI_Allerton} in the scalar setting for the vector one.

The problem of control under communication constraints with side information has recently attracted much interest \cite{Kostina_SI_Allerton,Tanaka_SI,STAVROU_SI,lev_khina_LB,lev2020schemes}. In \cite{Kostina_SI_Allerton}, a scalar version of the problem in Fig. \ref{fig:setting} was solved. In \cite{Tanaka_SI}, a slightly less general problem than Fig.~\ref{fig:setting} was considered. They conjectured that a linear, memoryless policy is optimal and provided a semidefinite programming solution. The conjecture and the SDP formulation are subsumed in the conference version of the current paper \cite{sabag_cdc}, published prior to \cite{Tanaka_SI}. Additionally, \cite{Tanaka_SI} shows that the conditional directed information is within a constant gap from the operational problem of variable-length coding with side information available to the controller and the encoder. This is obtained by constructing a practical coding scheme and analyzing its performance. In \cite{STAVROU_SI}, the rate-distortion counterpart of the control problem studied here is considered. It is shown that if the optimal policy is assumed to be linear and the LQG cost admits an upper bound at all times, a simple optimization problem can be realized for the corresponding rate-distortion problem. The result presented below in Theorem~\ref{th:structure} confirms the optimality of the policy conjectured in \cite{Tanaka_SI} and of the linear policy assumed in \cite{STAVROU_SI}. It should be remarked that the objective considered in \cite{Tanaka_SI,STAVROU_SI} and that in the current paper is the conditional directed information, which is a lower bound to the operational problem in the case of a fixed rate or in the case of a variable rate and prefix-free codebooks. In \cite{lev_khina_LB}, it is shown that the directed information is a tighter lower bound but, it is also illustrated that a Gaussian policy does not attain its minimum and therefore, it is not clear whether a computable form of the directed information can be obtained. Finally, \cite{lev2020schemes} studied coding schemes for the scalar LQG setting with a Gaussian communication channel based on the joint source channel schemes in \cite{Tuncel_SI,kochman_modulo_SI}.

The remainder of this paper is organized as follows. Section \ref{sec:setting} introduces the notation, setting and the problem definition. Section \ref{sec:main} presents our main results and Section \ref{sec:proofs} provides their proofs. Section \ref{sec:numerical} presents numerical examples.

\section{The setting and problem definition} \label{sec:setting}
A linear dynamical system is described by
\begin{align}\label{eq:setting}
  \vx_{t+1} = A_t\vx_t+B_t\vu_t + \v{w}_t \ \ \ t\ge1,
\end{align}
where $\v{w}_t\sim \mathcal N(0,W_t)$ are mutually independent. The initial state $\vx_1$ is distributed according to $P_{1|0}$ and is independent of $\v{w}_t$. A noisy measurement of the state is available to the controller,
\begin{align}\label{eq:setting_SI}
    \vy_{t} = C_t\vx_t+ \vv_t,
\end{align}
with $\vv_t\sim \mathcal N(0,V_t)$. For a fixed time-horizon $T$, the LQG quadratic cost is defined as
\begin{align}\label{eq:def_lqgcost}
        J(\vx^{T+1},\vu^T) = \left[\sum_{t=1}^T \vx_{t+1}^{\ast} Q_t \vx_{t+1} + \vu_t^{\ast} R_t\vu_t\right],
\end{align}
with $Q_t\succeq0$ and $R_t\succ0$, and superscripts denote vectors starting at time $t=1$, e.g., $\vx^{T+1}\triangleq (\vx_1,\dots,\vx_{T+1})$.

The objective is to design a system such that the LQG cost does not exceed a cost target denoted by $\Gamma$. Naturally, if the measurements $\vy_t$ are sufficient to attain $\Gamma$, the classical solution to the LQG problem is satisfactory, and there is no need to expand. In the other extreme, the LQG cost cannot be reduced below the LQG cost attained by a fully observer, i.e., $\vy_t=\vx_t$. Our interest lies in the scenario where $\Gamma$ is below the optimal LQG cost attainable with the partial observer \eqref{eq:setting_SI} but above the optimal LQG cost attainable with the full observer. In this case, the introduction of a communication/information link (see the dashed line in Fig.~\ref{fig:setting}) between a full observer (encoder) and a controller (a decoder) will help to attain the desired LQG cost $\Gamma$.

The encoder is characterized by the set of stochastic mappings that can be compactly represented by the causal conditioning
\begin{align}\label{eq:enc}
    P(\vf^T||\vx^T)&\triangleq \prod_{t=1}^T P(\vf_t|\vf^{t-1},\vx^t).
\end{align}
Similarly, the decoder (controller) is a causally conditioned probability distribution
\begin{align}\label{eq:dec}
    P(\vu^T||\vf^T,\vy^T)&\triangleq \prod_{t=1}^T P(\vu_t|\vu^{t-1},\vf^t,\vy^t).
\end{align}

By the construction, the encoder-decoder pair satisfies at all times
\begin{align}\label{eq:policy}
    &P(\vu_t,\v{f}_t|\v{f}^{t-1},\vu^{t-1},\vx^t,\vy^t) \nn\\
    &\ \ = P(\vu_t|\vu^{t-1},\v{f}^{t},\vy^t)P(\v{f}_t| \vx^t,\v{f}^{t-1}).
\end{align}
The overall joint distribution can be summarized using the one-step update
\begin{align}
    &P(\vx_t,\vy_t,\vu_t,\v{f}_t|\vx^{t-1},\vy^{t-1},\v{f}^{t-1},\vu^{t-1})\nn\\
    &\ = P(\vy_t,\vx_t|\vx_{t-1},\v{u}_{t-1})P(\vu_t,\v{f}_t|\v{f}^{t-1},\vu^{t-1},\vx^t,\vy^t),
\end{align}

The communication resources are measured by the directed information from the encoder to the controller causally conditioned on the partial observations at the controller \cite{Massey90,Kramer98}:
    \begin{align}\label{eq:def_DI}
     I(\vx^T\to\vf^T||\vy^T) = \sum_{t=1}^T I(\vx^t;\vf_t|\vf^{t-1},\vy^{t}),
    \end{align}
where $I(X;Y|Z)$ is the mutual information between $X$ and $Y$ conditioned on $Z$.

The objective of this paper is to solve the optimization problem:
\begin{align}\label{eq:Optimization_def}
   &\min  I(\vx^T\to\v{f}^T||\vy^T)\nn\\
   &\ \ \text{s.t.} \ J(\vx^{T+1},\vu^T)\le \Gamma,
\end{align}
where the minimum is over policies of the form \eqref{eq:policy}.

When the measurement $\vy_t$ is absent, the optimization problem in \eqref{eq:Optimization_def} simplifies to the directed information ~$I(\vx^T\to\vf^T)$ that was investigated in \cite{silva_framework,Tanaka_SDP}. To see that the conditional directed information measures the information encapsulated at the encoding policy, assume that the $t$-th element in the conditional directed information satisfies:
\begin{align}
  I(\vx^t;\vf_t|\vf^{t-1},\vy^{t}) = I(\vx_t;\vf_t|\vf^{t-1},\vy^{t}).
\end{align}
Then, the right hand side extracts the state uncertainty at the controller with and without the encoding variable $\vf_t$, i.e., $I(\vx_t;\vf_t|\vf^{t-1},\vy^{t}) = h(\vx_t|\vf^{t-1},\vy^{t}) - h(\vx_t|\vf^{t},\vy^{t})$. Specifically, the difference reflects the fact that $\vf_t$ is costly while $\vy_t$ is a natural occurrence of the dynamical system without any cost. These arguments are formalized in Theorem \ref{th:structure} and Lemma \ref{lemma:covariances}. We will also show a relation between the optimal conditional directed information and the Kalman filtering theory with two independent measurements.

\section{Results}\label{sec:main}
This section presents our results. First, we provide a simple structure for the optimal policy in Theorem \ref{th:structure}. Then, we present preliminaries on Kalman filtering theory to express the directed information in its terms. We then provide a semidefinite programming formulation of the optimization problem and present the optimal system design. Finally, Section \ref{subsec:infinite} includes the formulation and the solution for the infinite-horizon problem.

\subsection{Optimal policy structure}
The first result is the optimal structure of the observer (encoder) and controller (decoder) policies:
\begin{theorem}[Optimal policy structure]\label{th:structure}
An optimal policy for the optimization problem in \eqref{eq:Optimization_def} is given by
\begin{align}\label{eq:th_structure_1}
    \vf_t &= D_t\vx_t + \v{m}_t, \nn\\
    \vu_t &= -K_t\E[\vx_t|\vf^{t},\vy^{t}],
\end{align}
where $\v{m}_t\sim\mathcal{N}(0,M_t)$ is independent from $(\vx^t,\vy^t,\vf^{t-1})$ and $K_t$ is a constant given by the LQR controller (see \eqref{eq:control_gain}, below).

Moreover, the knowledge of the measurements $\vy^t$ at the encoder does not reduce the optimal directed information control problem in \eqref{eq:Optimization_def}.
\end{theorem}
The theorem simplifies significantly the maximization domain from the general policy in \eqref{eq:policy} to the set $\{(D_t,M_t)\}_{t=1}^T$. The encoding rule reveals that $\vf_t$ reduces the communication resources by introducing an additive noise to the state observation. We emphasize that our problem formulation does not impose any structural constraints onto the encoding policy such as linear, memoryless, or following a Gaussian distribution. The control signal $\vu_t$ is the standard LQG certainty equivalence controller. Thus, similar to the scalar case in \cite{Kostina_SI_Allerton}, the separation between the control gain and the estimation is preserved in our setting. The proof of Theorem~\ref{th:structure} appears in Section~\ref{sec:proofs}.

Theorem \ref{th:structure} extends \cite[Th. $1$]{Tanaka_SDP} and recovers it when $\vy_t$, the observation, is absent. The extension of \cite{Tanaka_SDP} to our setting is not trivial (see e.g., \cite{Tanaka_SI,STAVROU_SI} for progress on that problem), and involves the study of a relaxed optimization problem where, at time $t$, the vector $\vy^t$ is also available to the encoder. For this relaxed optimization problem, we show that the optimal policy is of the form \eqref{eq:th_structure_1}. In other words, even if the side information is available at the encoder, it cannot reduce the conditional directed information. This is consistent with the observation made in \cite{Kostina_SI_Allerton} in the context of the scalar system.

\subsection{Kalman filter with two (independent) measurements}
As is evident from the optimal structure in Theorem \ref{th:structure}, the encoding function $\vf_t$ is a noisy measurement of the system state, and its additive noise is independent of the other measurement $\vy_t$. Thus, the optimal system has a structure of an LQG setting with two independent observations. However, for the purpose of optimizing the communication resources, $\vf_t$ has a cost, while $\vy_t$ is a natural occurrence of the system. In this section, we provide short preliminaries on Kalman filtering and present the conditional directed information in Kalman filtering terms.

Following a standard convention, we denote the error covariance matrices with respect to both measurements  $\vy_t$ and $\vf_t$ as
\begin{align}\label{eq:def_error_1}
    P_{t|t-1} &\triangleq \text{Cov}(\vx_t - \E[\vx_t|\vf^{t-1},\vy^{t-1},\vu^{t-1}])\nn\\
        P_{t|t} &\triangleq \text{Cov}(\vx_t - \E[\vx_t|\vf^t,\vy^t,\vu^{t-1}]).
\end{align}
Since the communication resources should be measured with respect to the observation $\vf_t$ only, we define the intermediate error covariance matrix corresponding to the prediction error after observing $\vy_t$ only:
\begin{align}\label{eq:def_error_2}
        P^+_{t|t-1} &\triangleq \text{Cov}(\vx_t - \E[\vx_t|\vf^{t-1},\vy^{t},\vu^{t-1}]).
\end{align}

\begin{figure*}[b]
\rule{17.6cm}{0.4pt}
  \begin{align}\label{eq:omega_constraint}
         \Omega_1&\triangleq \begin{bmatrix} P_{1|0} - P_{1|1}& P_{1|0}C_1^{\tr}\\C_1P_{1|0}& C_1P_{1|0}C_1^{\tr} +V_1
\end{bmatrix} \nn\\
\Omega_t&\triangleq\begin{bmatrix} (A_{t-1}P_{t-1|t-1}A_{t-1}^{\tr} +W_{t-1}) - P_{t|t}& (A_{t-1}P_{t-1|t-1}A_{t-1}^{\tr} +W_{t-1})C_t^{\tr}\\C_t(A_{t-1}P_{t-1|t-1}A_{t-1}^{\tr} +W_{t-1})& C_t(A_{t-1}P_{t-1|t-1}A_{t-1}^{\tr} +W_{t-1})C_t^{\tr} +V_t
\end{bmatrix} &\text{for} \ \ \ t=2, \dots, T.
  \end{align}
\end{figure*}
The following lemma formalizes several relations between the error covariances.
\begin{lemma}[Error covariance matrices]\label{lemma:covariances}
Let $P_{1|0}$ be the covariance matrix of $X_1$. Then, for a fixed policy $\{(D_t,M_t)\}_t$, the error covariance matrices can be updated as
\begin{subequations}
\begin{align}
    (P^+_{t|t-1})^{-1} &= P_{t|t-1}^{-1}+\snr^Y_t\\
    P_{t|t} &= ((P^+_{t|t-1})^{-1} + \snr^F_t)^{-1}\nn\\
            &= (I-L^F_tD_t)P^+_{t|t-1} \label{eq:cov_update_b}\\
    P_{t+1|t}   &=  A_{t}P_{t|t}A_{t}^{\tr} + W_{t},
\end{align}
\end{subequations}
where $L^F_t = P^+_{t|t-1}D_t^{\tr}(D_tP^+_{t|t-1}D_t^{\tr} + M_t)^{-1}$,  $\snr^F_t = D_t^{\tr}M_t^{-1}D_t$ and $\snr^Y_t= C_t^{\tr}V_t^{-1}C_t$. 
\end{lemma}
The identities are standard in Kalman filtering theory, and their proofs are omitted. It now follows that the directed information can be expressed as
\begin{align}\label{eq:th_kalman}
    I(\vx^T\to\v{f}^T||\vy^T) &= h(\vx_t|\vy^t,\vf^{t-1}) - h(\vx_t|\vy^t,\vf^{t})\nn\\
    &= \frac1{2}\sum_{t=1}^T \log \det (I-L^F_tD_t).
\end{align}
Note that the matrix $(I-L^F_tD_t)$ is the multiplicative term of the error reduction when computing $P_{t|t}$ from $P^+_{t|t-1}$. Therefore, the conditional directed information measures the reduction in error covariance with respect to $\vf_t$ only, as desired.
\subsection{SDP formulation}
Despite the elegant representation of the objective function in \eqref{eq:th_kalman}, it is not clear whether \eqref{eq:Optimization_def} can be formulated as a convex optimization since its inverse includes a product of two optimization variables ${(I-L^F_tD_t)^{-1} = I + P^+_{t|t-1}\snr^F_t}$. Our next result shows a convex optimization formulation for \eqref{eq:Optimization_def}.


\begin{theorem}[SDP formulation]\label{th:SDP}
For a fixed $P_{1|0}\succeq0$, the optimization problem~\eqref{eq:Optimization_def} can formulated as the convex optimization
\begin{align}\label{eq:th_main}
        &\inf_{\{P_{t|t},\Pi_t\}_{t=1}^T}~ \Lambda - \frac{1}{2}\sum_{t=1}^{T-1} \log \det(I+(A_tP_{t|t}A_t^{\tr} + W_t)\snr^Y_t) \nn\\
        & \ \ - \frac{1}{2}\sum_{t=1}^{T} \log \det \Pi_t \nn\\
        &\text{s.t. } \quad \Trace(\Phi_1 P_{1|0}) + \sum_{t = 1}^T \Trace\left(\Theta_t P_{t|t}\right) + \Trace(S_tW_t) \leq \Gamma, \nn\\
        &\quad\begin{bmatrix} P_{t|t} - \Pi_t& P_{t|t}A_t^{\tr}\\A_tP_{t|t}& A_tP_{t|t}A_t^{\tr} +W_t
        &\end{bmatrix}\succeq0,\Pi_t\succ0, \ \ t< T\nn\\
        &\quad P_{T|T} = \Pi_T \succeq0,\nn\\
        &\quad\Omega_t\succeq0,\  t = 1,\dots, T \ \text{(Eq. \eqref{eq:omega_constraint} below)},
\end{align}
where the constant matrices $\Phi \triangleq A^{\tr}_{t} S_{t}A_{t} - K_{t}^{\tr}(B^{\tr}_{t}S_{t}B_{t} + R_{t}) K_{t} $ and $\{\Theta_t\}_{t=1}^T$ can be computed from  \eqref{eq:control_gain} below, and the constant $\Lambda$ is given by
\begin{align}
    \Lambda &= - \frac1{2}\log \det( P_{1|0}^{-1}+\snr_1^Y) + \frac1{2}\sum_{t=1}^{T-1} \log \det W_t.
\end{align}
\end{theorem}

The optimization problem in Theorem \ref{th:SDP} is convex optimization with respect to the decision variables $(P_{t|t},\Pi_t)$, and can be solved using standard solvers, e.g., \cite{boyd_detmax,cvx,yalmip}\footnote{Some solvers require to write the determinant of $I + (A_t P_{t|t} A^{\tr} + W_t) \snr^Y_t$ in a symmetric form using Sylvester's determinant theorem.}. It will be shown in the proof of Theorem \ref{th:SDP} in Section \ref{sec:proofs} below that the auxiliary decision variable $\Pi_t$ evaluated at the optimal point is equal to $(P_{t|t}^{-1} + A_t^{\tr}W_t^{-1}A_t)^{-1}$. However, it is necessary to introduce this variable in order to convert the objective to have a standard convex form. Then, the equality constraint resulting from the change of variable $\Pi_t=(P_{t|t}^{-1} + A_t^{\tr}W_t^{-1}A_t)^{-1}$ can be (optimally) relaxed to an inequality that is equivalent to the LMI above. The optimization problem extends \cite[Th. $1$]{Tanaka_SDP} to the case where the LQG measurement $\vy_t$ is available to the controller, and recovers it by choosing $C_t = \snr_t^Y=0$. In this case, the constraints on $\Omega_t$ simplify to $(A_{t-1}P_{t-1|t-1}A_{t-1}^{\tr} +W_{t-1}) - P_{t|t}\succeq 0$ and $P_{1|0}\succeq P_{1|1}$.

\subsection{System design}\label{subsec:system_design}
In this section, we construct a three-steps realizable policy using the results from the previous section..
\subsubsection{The controller gain}
The controller gains are independent of the measurements and the variables from the optimization problem. The gains can be computed from a backward Riccati recursion, with the initial condition $S_T = Q_T$, as
\begin{align}\label{eq:control_gain}
    S_{t-1} &= A^{\tr}_{t} S_{t}A_{t} - K_{t}^{\tr}(B^{\tr}_{t}S_{t}B_{t} + R_{t}) K_{t} + Q_{t-1} \nn\\
    K_t &= (B^{\tr}_tS_{t}B_t + R_t)^{-1}B^{\tr}_t S_{t}A_t\nn\\
    \Theta_t&= K_t^{\tr} (B^{\tr}_tS_{t}B_t + R_t) K_t.
\end{align}
\subsubsection{Covariance matrices} Given the sequence $\{\Theta_t\}_{t=1}^T$, the optimal $\{P_{t|t}\}_{t=1}^T$ can be determined from the convex optimization problem in Theorem \ref{th:SDP}, and one can compute
$$\snr_t^F = P^{-1}_{t|t} - (A_{t-1}P_{t-1|t-1}A^{\tr}_{t-1}+W_{t-1})^{-1} - \snr_t^Y.$$
An application of the SVD decomposition $\snr^F_t =D_t^{\tr}M^{-1}_tD_t$ determines the parameters $\{(D_t,M_t)\}_{t=1}^T$ of the optimal policy in Theorem \ref{th:structure}.

\subsubsection{Kalman filter}
The Kalman gain is defined as
\begin{align}
  L_t &= P_{t|t-1}H_t^{\tr}(H_tP_{t|t}H_t^{\tr} + N_t)^{-1},
\end{align}
where $H_t\triangleq\begin{bmatrix}
C_t \\
D_t
\end{bmatrix} , N_t\triangleq\begin{bmatrix}
V_t &0\\
0&M_t
\end{bmatrix}$.

The Kalman update is done in two steps:
\begin{align}\label{eq:Kalman_update}
    \hat{\vx}_{t+1|t} &= A_t\hat{\vx}_{t} + B_t\vu_t\nn\\
    \hat{\vx}_{t} &= \hat{\vx}_{t|t-1} + L_t\begin{bmatrix}\vy_t - C_t\hat{\vx}_{t|t-1}\\ \vf_t- D_t\hat{\vx}_{t|t-1}\end{bmatrix},
\end{align}
where the control signal is $\vu_t = -K_t\hat{\vx}_t$.
\subsection{The infinite-horizon setting}\label{subsec:infinite}
In this section, we formulate and solve the optimization problem  \eqref{eq:Optimization_def} in the infinite-horizon regime. In this scenario, we  consider time-invariant systems, i.e., $A_t= A$, $B_t= B$, $W_t= W$, $C_t= C$, $V_t= V$ and time-invariant cost matrices $Q_t = Q$, $R_t=R$. The optimization problem is defined as:
\begin{align}\label{eq:Optimization_def_inf}
   &\inf \limsup_{T\to\infty} \frac1{T}( \vx^T\to\v{f}^T||\vy^T)\nn\\
   &\ \ \text{s.t.} \ \limsup_{T\to\infty} \frac1{T} J(\vx^{T+1},\vu^T)\le \Gamma,
\end{align}
where the infimum is taken with respect to the sequence of stochastic policies given in \eqref{eq:policy}.

The solution structure is similar to the finite-horizon solution in Theorem \ref{th:SDP}. In particular, we construct a controller based on a solution to a convex optimization problem. We begin with the controller description.
\subsubsection{Controller gain}
Assume that $(A,B)$ is stabilizable and $(A,Q^{1/2})$ is observable on the unit circle. Then, we define $\bar{S}$ to be the unique stabilizing solution for the Riccati equation
\begin{align}\label{eq:S_bar_riccati}
&A^{\tr}SA-S-A^{\tr}SB(B^{\tr}SB+R)^{-1}B^{\tr}SA+Q=0.
\end{align}
By having the stabilizing solution, we can present the SDP-based system design in the infinite-horizon regime.
\begin{theorem}\label{th:SDP_infinite}
If the pair $(A,B)$ is stabilizable and the pair $(A,Q^{1/2})$ is observable on the unit circle, the infinite-horizon optimization problem~\eqref{eq:Optimization_def_inf} can be formulated as the convex optimization
\begin{align}\label{eq:infinite_optimization}
        &\min_{P,\Pi}~ \frac1{2} \log \det W - \frac{1}{2} \log \det(I+\snr^Y(APA^T+W)) \nn\\
        & \ - \frac{1}{2} \log \det \Pi \nn\\
        \text{s.t. } &\quad \Trace\left(\Theta P\right) + \Trace\left(W \bar{S}\right)\leq D, \nn\\
        &\quad\begin{bmatrix} P - \Pi& PA^{\tr}\\AP& APA^{\tr} +W
        &\end{bmatrix}\succeq0,\quad\Pi\succ0.\\
        &\quad \begin{bmatrix} APA^{\tr} + W - P& (APA^{\tr} +W)C^{\tr}\\C(APA^{\tr} +W)& C(APA^{\tr} +W)C^{\tr} +V
\end{bmatrix}\succeq 0,\nn
\end{align}
where $\bar{S}$ is given in \eqref{eq:S_bar_riccati}, and $\Theta= K^{\tr}(B^{\tr}\bar{S}B+R)K$.

Moreover, let $P$ be the optimal solution in \eqref{eq:infinite_optimization} and compute
\begin{align}\label{eq:policy_infinite}
  \snr^F &= P^{-1} - (APA^{\tr}+W)^{-1} - \snr^Y,
\end{align}
and its SVD decomposition as $\snr^F = D^{\tr}M^{-1}D$. Then, optimal time-invariant encoder and decoder are given by
\begin{align}
    \vf_t&= D\vx_t + \v{m}_t\nn\\
        \vu_t&= -K \hat{\vx}_t,
\end{align}
where $\v{m}_t\sim N(0,M)$, $K=(B^{\tr}\bar{S}B+R)^{-1}B^{\tr}\bar{S}A$, and $\hat{\vx}_t$ is computed recursively using the Kalman filter in \eqref{eq:Kalman_update}.
\end{theorem}
Theorem \ref{th:SDP_infinite} shows that the optimization problem in the infinite-horizon regime is computationaly simpler than the finite-horizon regime solved in Theorem \ref{th:SDP}. In the proof of Theorem \ref{th:SDP_infinite}, Theorem \ref{th:structure} is used for the structure of the optimal policy, however, it is interesting to note that we also show that a time-invariant law is optimal while in Theorem \ref{th:SDP} the optimal policy is time-varying. The main idea to show this property is the convexity of the objective. In particular, one can use Jensen's inequality to show that the evaluation of the objective at the convex combination of the decision variables is smaller than the averaged sum of objectives at all times. This fact can be exploited in the infinite-horizon regime to show that the convex combination of the decision variables satisfies the \emph{stationary constraints} presented in Theorem \ref{th:SDP_infinite}. The proof of Theorem \ref{th:SDP_infinite} is given in Section \ref{subsec:proof_infinite}.

\section{Examples}\label{sec:numerical}
\subsection{Side information reduces the minimal directed information}
\begin{figure}[t]
    \centering
    \includegraphics[width = 0.48\textwidth]{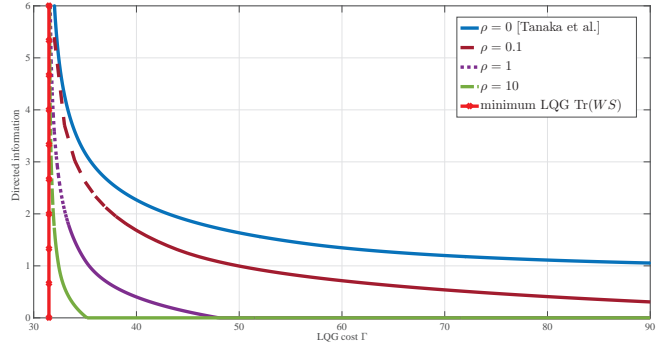}
    \caption{The trade-offs between the conditional directed information and the LQG cost when the SNR of the side information varies.}
    \label{fig:SDPCompare}
\end{figure}
In this section, we study a numerical example to show the benefits of side information and discuss the trade-offs between communication resources and control performance. We set the matrices $A, B, W, Q, R$ to be the same as those in \cite[Sec. V]{Tanaka_SDP}
{ \begin{align}
    A &=
    \begin{pmatrix}
    0.12 &0.63 &-0.52& 0.33\\
    0.26& -1.28 &1.57& 1.13\\
    -1.77& -0.30& 0.77& 0.25\\
    -0.16& 0.20& -0.58& 0.56
    \end{pmatrix}, \nn\\
    B &=
    \begin{pmatrix}
    0.66 &-0.58& 0.03 &-0.20\\
    2.61 &-0.91 &0.87 &-0.07\\
    -0.64& -1.12 &-0.19 &0.61\\
    0.93 &0.58 &-1.18 &-1.21
    \end{pmatrix}, \nn\\
    W&= \begin{pmatrix}
     4.94 &-0.10& 1.29& 0.35\\
    -0.10 &5.55 & 2.07 & 0.31\\
     1.29 &2.07 & 2.02 & 1.43\\
     0.35 & 0.31 & 1.43 & 3.10
    \end{pmatrix},
\end{align}}\normalsize
and the cost matrices $Q,R$ are set to be identity matrices.

We start by studying an LQG system in which the side information to the decoder is given by $C = I$ and $V = \frac{1}{\rho}I$ with $\rho > 0$, so that $\snr^{Y} = \rho I$. For each $\rho = 0.1$, $\rho= 1$ and $\rho = 10$, we solve~\eqref{eq:infinite_optimization} for each LQG cost constraint $\Gamma$ in the range $\Gamma\in [30, 90]$ and plot the optimal value of~\eqref{eq:infinite_optimization} as a function of $\Gamma$ in Fig.~\ref{fig:SDPCompare}. The case without side information studied in~\cite{Tanaka_SDP} can be equivalently viewed as the case with $\rho = 0$.

In Fig. \ref{fig:SDPCompare}, we can see that for any fixed $\Gamma$, the minimal conditional directed information decreases as $\rho$ (the signal-to-noise ratio of the side information) increases. The red vertical line corresponds to the minimal cost that can be attained with clean observation available at the controller. The intersection with the LQG constraint axis corresponds to the LQG cost that is achieved without communication, that is, using the side information only. It is also interesting to note that a fixed information level, the gain due to the presence of $\vy_t$ increases for an increasing control cost.

In all curves with side information, the minimal directed information converges to zero as the LQG cost increases to infinity. However, in the case without side information, the curve  converges to some constant known as the minimal rate needed to stabilize the system $A$ \cite{NairEvans04}. This rate can be computed as $R = \sum_i \log_2 \max \{1,|\lambda_i(A)|\}$, where $\lambda_i(\cdot)$ denotes the $i$th eigenvalue of its argument. The fact that the curves converge to zero follow from the detectability of the pair $(A,CV^{1/2})$ (indeed, $CV^{1/2}=\rho^{-1/2}I$ is a full-rank so that the pair is observable). We proceed to study a scenario in which the side information implies that the pair is not detectable.

\begin{figure}[t]
    \centering
    \includegraphics[width = 0.48\textwidth]{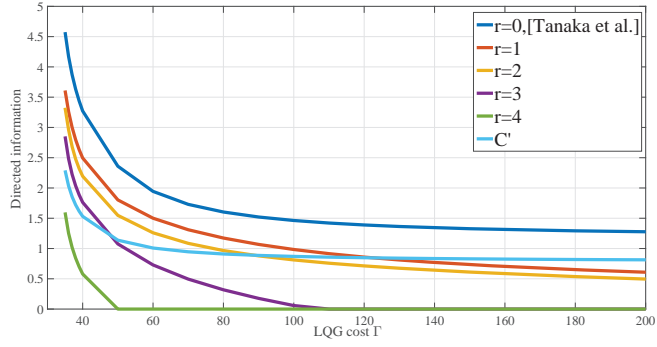}
    \caption{The trade-offs between the conditional directed information and the LQG cost when the SNR of the side information varies.}
    \label{fig:partial_obs}
\end{figure}
Here, we fix the side information variance to be the identity matrix $V =I$ (i.e., $\rho=1$), but change the observability matrix $C$ according to two scenarios. In the first, the matrix $C$ has dimensions $r\times 4$ for $0\le r\le 4$, and is given by $C(r) = [0_{r\times(4-r)},I_r]$. Clearly, if $r=0$, there is no side information, and if $r=4$ it is the full-observable matrix studied in Fig. \ref{fig:SDPCompare} with $\rho=1$. In the other case, we carefully choose $C$ to be orthogonal to one of the \emph{unstable eigenvectors} of $A$, i.e., the eigenvector whose corresponding eigenvalue is $\lambda_1 = -1.7124$. One choice of such a matrix is~$C' = \begin{pmatrix}
1&1&1&3.75\\
2.11&1&1&1\\
1&1&0&4.56\end{pmatrix}.$

In Fig. \ref{fig:partial_obs}, the minimal directed information is plotted as a function of the LQG cost $\Gamma$. As expected, it can be observed that the communication resources are decreasing as the side information dimension is increasing. For all observability matrices $C(r)$ with $1\le r\le 4$, the curves tend to zero as the cost $\Gamma$ grows to $\infty$. On the other hand, the curves that correspond to $r=0$ from \cite{Tanaka_SDP}, and the observability matrix $C'$ tend to a constant when the cost is large. This constant can be calculated as the minimal rate needed to stabilize the system. In the blue curve, it is $R = \sum_i \log_2 \max \{1,|\lambda_i(A)|\} = 1.1685$ and for $C'$ it is  $R' = \log_2 |\lambda_1(A)| = 0.776$ where $\lambda_1$ is the only unstable eigenvalue that cannot be observed via $C'$.

\subsection{Scalar systems}
For scalar systems, without the LQG measurement $(C=V=0)$, the solution to ~\eqref{eq:infinite_optimization} \cite{TaikondaSahaiMitter,Tanaka_SDP,silva_framework} is
\begin{align}\label{eq:scalr_NoSI}
    \frac{1}{2}\log\left(A^2 + \frac{W\Theta}{\Gamma - W\bar{S}}\right),\quad \forall~ \Gamma> W\bar{S},
\end{align}
where $\bar{S}$ is the unique solution to the Riccati equation and can be solved in closed-form as
\begin{align}\label{eq:scalarSbar}
    \bar{S} = \frac{(A^2 + B^2 - 1) + \sqrt{(A^2 + B^2 - 1)^2 + 4B^2}}{2B^2}.
\end{align}
In the following result, we provide a closed-form for the scalar problem. The proof is in~Section~\ref{sec:scalarSDP} below.
\begin{corollary}\label{co:scalarSDP}
When $A, B, W, C, V$ are scalars, $Q = R = 1$ and $|A|>1$, the optimal value of the optimization~\eqref{eq:infinite_optimization} is
\begin{align}\label{eq:scalar_sol1}
&\frac1{2} \log \left(A^2 + \frac{W\Theta}{\Gamma - W\bar{S}}\right) \nn\\
&\ - \frac{1}{2} \log \left( 1+ \snr^Y\left(W+\frac{A^2(\Gamma-W\bar{S})}{\Theta}\right)\right),
\end{align}
when $W\bar{S} < \Gamma \leq W\bar{S} + \Theta P^\star$; and is 0 when $\Gamma > W\bar{S} + \Theta P^\star$, where $P^\star$ is the unique positive solution to the quadratic equation
\begin{align}\label{eq:scalarPstar}
   A^2\snr^Y P^2 + (1 - A^2 + \snr^YW)P - W = 0
\end{align}
and $\bar{S}$ is given in~\eqref{eq:scalarSbar} and $\Theta = \frac{(AB\bar{S})^2}{1 + B^2\bar{S}}$.
\end{corollary}
By comparing \eqref{eq:scalr_NoSI} and \eqref{eq:scalar_sol1}, the information gain due to the presence of the LQG measurement is the non-negative expression
\begin{align}
  \frac{1}{2} \log \left( 1+ \snr^Y\left(W+\frac{A^2(\Gamma-W\bar{S})}{\Theta}\right)\right).
\end{align}
Note that the gain is an increasing function of $\snr^Y$. Also, the gain is upper bounded by \eqref{eq:scalr_NoSI} which is achieved with equality when $\Gamma = W\bar{S} + \Theta P^\star$ since $P^\star$ satisfies $1+ \snr^Y(W+A^2P^\star) = A^2 + \frac{W}{P^\star}$ (see Eq. \eqref{eq:scalarPstar}).
\begin{remark}
In \cite{Kostina_SI_Allerton}, the rate distortion problem which corresponds to the control problem studied in this paper has been solved for the scalar case. To reveal  \cite[Th. $7$]{Kostina_SI_Allerton} from Corollary \ref{co:scalarSDP}, let $d\triangleq \frac{\Gamma - W\bar{S}}{\Theta}$ in order to write \eqref{eq:scalar_sol1} as
\begin{align}
  - \frac1{2} \log d \left(\snr^Y + \frac1{A^2 d+W}\right).
\end{align}
\end{remark}




\section{Proofs}\label{sec:proofs}
In this section, we prove our results. We start with Theorem~\ref{th:structure} on the optimal policy structure.
\subsection{Proof of Theorem \ref{th:structure} (Optimal policy structure)}
The proof follows from the following claims that will be shown consecutively thereafter.
\begin{enumerate}
    \item Instead of minimizing over stochastic kernels $P(\vu_t|\vu^{t-1},\vf^t,\vy^t)$ in \eqref{eq:dec}, it is sufficient to minimize over $\vu_t$ that is a deterministic function of $\vf^t,\vy^t$.
    \item The minimization domain is relaxed by allowing encoders of the form $P(\vf_t|\vy^t,\vx^t,\v{f}^{t-1})$ instead of $P(\vf_t|\vx^t,\v{f}^{t-1})$ (in \eqref{eq:enc}). That is, the new encoder has additional access to the observation $\vy_t$.
    \item It is sufficient to minimize the relaxed optimization problem over $P(\vf_t|\vy^t,\vx_t,\v{f}^{t-1})$, i.e., to let the encoder depend on $\vx_t$ rather the tuple $\vx^t$.
    \item It is sufficient to minimize the relaxed optimization problem over Gaussian encoder outputs, i.e,
\begin{align}\label{eq:f_full_dependence}
    \vf_t &= D_t\vx_t + \lambda_t \vy^{t} + \gamma_t \vf^{t-1} + \v{m}_t,
\end{align}
where $\v{m}_t\sim\mathcal{N}(0,M_t)$.
\item It is sufficient to minimize the relaxed optimization problem over
    \begin{align}\label{eq:f_proof_memoryless}
        \vf_t &= D_t\vx_t + \v{m}_t.
    \end{align}
\item  The optimal control is $\vu_t = -K_t\E[\vx_t|\v{f}^{t},\vy^t]$, where $K_t$ is the control gain.
\end{enumerate}
By claim $5$, the minimizer of the relaxed optimization problem is in the original minimization domain \eqref{eq:policy}. Thus, both optimization problems have a common minimizer, and $\vu_t$ is a composition of a Kalman filter and certainty equivalence controller.

\underline{Claim $1$:}
From the functional representation lemma \cite{ElGamal}, one can write $\vu_1 = \Lambda(\v{f}_1,\vy_1,W_1)$ for some deterministic function $\Lambda(\cdot)$ and random variable $W_1$ that is independent of $(\vf_1,\vx_1,\vy_1)$. Let $\tilde{\vf}_1 \triangleq (\vf_1,W_1)$, and note that~$I(\vx_1;\vf_1|\vy_1) = I(\vx_1;\tilde{\vf_1}|\vy_1)$.
Moreover, the joint distribution of $\vu_1$ and $\vx_1$ is unaffected by absorbing the controller's randomness to the encoder (stochastic) mapping so the LQG cost remains the same.
This procedure can be inductively repeated to de-randomize $\vu_t$ at all times.\\
\underline{Claim $2$:} Trivial, since the minimization domain is increased.\\
\underline{Claim $3$:} Consider a simple lower bound on the objective function,
\begin{align}\label{eq:proof_obj_LB}
    I( \vx^T\to\v{f}^T||\vy^T) &= \sum_t I(\vx^t;\v{f}_t|\vy^t,\v{f}^{t-1})\nn\\
    &\ge \sum_t I(\vx_t;\v{f}_t|\vy^t,\v{f}^{t-1}).
\end{align}
For a fixed sequence of deterministic mappings characterizing $\vu_t$, the lower bound \eqref{eq:proof_obj_LB} and the LQG cost are fully determined by $\{P(\vx_t,\vy^t,\vf^t)\}_{t\ge1}$.

We will now show by induction that $P(\vx_t,\vy^t,\vf^t)$ is determined by $\{P(\vf_i|\vy^i,\vx_i,\v{f}^{i-1})\}_{i\le t}$. For~$t=1$, this claim is trivial. For the inductive step, assume that $P(\vx_{t-1},\vy^{t-1},\vf^{t-1})$ is determined by $\{P(\vf_i|\vy^i,\vx_i,\v{f}^{i-1})\}_{i<t}$. Now, consider
\begin{align*}
P(\vx_t,\vy^t,\vf^t)
&= P(\vf_t|\vx_{t},\vy^{t},\vf^{t-1}) P(\vx_{t},\vy^{t},\vf^{t-1}),
\end{align*}
and note that $P(\vx_{t},\vy^{t},\vf^{t-1})$ can be written as
\begin{align*}
\int_{\vx_{t-1}}  P(\vy_t|\vx_t) P(\vx_t|\vx_{t-1},\vy^{t-1},\vf^{t-1})P(\vx_{t-1},\vy^{t-1},\vf^{t-1}),
\end{align*}
which is fixed by the sequence $\{P(\vf_i|\vy^i,\vx_i,\v{f}^{i-1})\}_{i<t}$
due to the measurement characteristics \eqref{eq:setting_SI}, the fact that $\vu_{t-1}$ is a deterministic function of $(\vy^{t-1},\vf^{t-1})$ and the induction hypothesis.\\
\underline{Claim $4$:}
First, the differential entropy from \eqref{eq:proof_obj_LB} is re-written as,
\begin{align}\label{eq:proof_entropy}
     & h(\vx_t|\vy^t,\v{f}^{t-1}) \nn\\
     &\ = h(\vy_t,\vx_t,\vx_{t-1}|\vy^{t-1},\v{f}^{t-1}) \nn\\
     &\ \ - h(\vy_t|\vy^{t-1},\v{f}^{t-1})- h(\vx_{t-1}|\vx_t,\vy^{t-1},\v{f}^{t-1}) \nn\\
     &\ = h(\vy_t,\vx_t|\vx_{t-1},\vu_{t-1}) + h(\vx_{t-1}|\vy^{t-1},\v{f}^{t-1}) \nn\\
     &\ \ - h(\vy_t|\vy^{t-1},\v{f}^{t-1})- h(\vx_{t-1}|\vx_t,\vy^{t-1},\v{f}^{t-1}).
\end{align}
We now lower bound the mutual information using \eqref{eq:proof_entropy},
\begin{align}\label{eq:proof_LB_2}
    &\sum_{t=1}^T I(\vx_t;\v{f}_t|\vy^t,\v{f}^{t-1})\nn\\
    &= \sum_{t=1}^T h(\vy_t,\vx_t|\vx_{t-1},\vu_{t-1}) + h(\vx_{t-1}|\vy^{t-1},\v{f}^{t-1}) \nn\\
    &- h(\vy_t|\vy^{t-1},\v{f}^{t-1})- h(\vx_{t-1}|\vx_t,\vy^{t-1},\v{f}^{t-1}) - h(\vx_{t}|\vy^{t},\v{f}^{t})\nn\\
    &= h(\vx_{1}) -  h(\vx_{T}|\vy^{T},\v{f}^{T})  + \sum_{t=1}^T h(\vy_t,\vx_t|\vx_{t-1},\vu_{t-1})  \nn\\
    &\ - h(\vy_t|\vy^{t-1},\v{f}^{t-1})- h(\vx_{t-1}|\vx_t,\vy^{t-1},\v{f}^{t-1})\nn\\
    &\ge h(\vx_{1}) -  h_{\PG}(\vx_{T}|\vy^{T},\v{f}^{T}) + \sum_{t=1}^T h(\vy_t,\vx_t|\vx_{t-1},\vu_{t-1}) \nn\\
    & - h_{\PG}(\vy_t|\vy^{t-1},\v{f}^{t-1})- h_{\PG}(\vx_{t-1}|\vx_t,\vy^{t-1},\v{f}^{t-1}),
\end{align}
where the inequality follows from $h_{P}(X)\le h_{\PG}(X)$ for any ${P}$ with the same covariance as ${\PG}$.

Conversely, the lower bound can be achieved by choosing $\vf_t$ with a Gaussian distribution. Specifically, for some fixed inputs $\{P(\vf_t|\vy^t,\vx_t,\v{f}^{t-1})\}_{t\ge1}$, a jointly Gaussian distribution is formed by borrowing the first and second order statistics of the joint. Let $D_t\vx_t + \lambda_t \vy^{t} + \gamma_t \vf^{t-1}$ be the linear minimum mean square estimator of $\vf_t$ and $\v{m}_t$ be its error covariance. Then, $\vf_t\sim\mathcal{N}(D_t\vx_t + \lambda_t \vy^{t} + \gamma_t \vf^{t-1},\v{m}_t)$, and it can be shown that the second-order statistics of $\{P(\vx_t,\vy^t,\vf^t)\}_{t\ge1}$ are unaffected since the relation between the random variables are all linear.

Finally, note that the LQG cost depends on  $\{P(\vu_t,\vx_t)\}_{t=1}^T$ via its second moments. Since the second moments of $\PG$ and $P$ are the same, the LQG cost is unaffected. To summarize, we showed that we may restrict the optimization domain to Gaussian inputs of the form $\vf_t = D_t\vx_t + \lambda_t \vy^{t} + \gamma_t \vf^{t-1} + \v{m}_t$ without loss of optimality.\\
\underline{Claim $5$:} By Claim $4$, the objective of the relaxed optimization problem can be written as:
\begin{align}
    &\sum_t I(\vx_t;D_t\vx_t + \lambda_t \vy^{t} + \gamma_t \vf^{t-1} + \v{m}_t|\vy^t,\v{f}^{t-1})\nn\\
    &\ \ \ = \sum_t I(\vx_t;D_t\vx_t + \v{m}_t|\vy^t,\v{f}^{t-1}),
\end{align}
where the equality follows since $\lambda_t \vy^{t} + \gamma_t \vf^{t-1}$ is constant when conditioned on $(\vy^t,\v{f}^{t-1})$. For the LQG cost, since $\vu_t$ is a deterministic function of $(\vy^t,\v{f}^{t-1})$, the effect of $\lambda_t \vy^{t} + \gamma_t \vf^{t-1}$ can be embedded into the controller's function.\\
\underline{Claim $6$:} In the previous steps, we showed that $\vf_t = D_t\vx_t + \v{m}_t$ is optimal. We now show that $\vu_t$ has no affect on the objective function. Therefore, for a fixed $\vf_t$, we have a classical LQG problem whose solution is just a Kalman filter with the control gain defined in \eqref{eq:control_gain}. Consider the objective:
\begin{align}
    &\sum_t I(\vx_t;\vf_t|\vy^t,\v{f}^{t-1})\nn\\
    & = \sum_t h(\vx_t|\vy^t,\v{f}^{t-1}) - h(\vx_t|\vy^t,\v{f}^{t})\nn\\
    &= \frac1{2}\sum_{t=1}^T \log \det(P^+_{t|t-1}) - \log \det(P_{t|t})\nn\\
    &\stackrel{(a)}= \frac1{2}\sum_{t=1}^T \log \det(P^+_{t|t-1}) + \log \det((P^+_{t|t-1})^{-1}+\snr^F_t)\nn\\
    &= \frac1{2} \sum_{t=1}^T \log \det(I + P^+_{t|t-1}\snr^F_t),
\end{align}
where $(a)$ follows from Lemma \ref{lemma:covariances}. Also, by Lemma \ref{lemma:covariances}, $P^+_{t|t-1}$ depends on the choice of $\snr_t^F$ only. Therefore, the objective is unaffected by $\vu_t$.
\subsection{Proof of Theorem \ref{th:SDP}}
Using Lemma \ref{lemma:covariances}, the optimization problem can be written as:
    \begin{align}\label{eq:proof_sdp_OP}
        &\min~\frac{1}{2}\sum_{t = 1}^T\log \det(P^+_{t|t-1}) - \log \det(P_{t|t}) \nn\\
        \text{s.t. } & \Trace(\Phi_1 P_{1|0}) + \sum_{t = 1}^T \Trace\left(\Theta_t P_t\right) + \Trace(S_tW_t) \leq \Gamma \nn\\
        &\ \ \ (P^+_{t|t-1})^{-1} = (P_{t|t-1})^{-1}+\snr^Y_t, t=1,\dots,T\nn\\
    &\ \ \ P^{-1}_{t|t}= (P^{+}_{t|t-1})^{-1} + \snr^F_t, t=1,\dots,T\nn\\
        &\ \ \ P_{t+1|t} =  A_{t}P_{t|t}A_{t}^{\tr} + W_{t},
    \end{align}
    where the minimization is over the covariance matrices $\{P^+_{t|t-1}\}_{t=1}^T$. We first rewrite the objective in a convex form, consider
\begin{align}
&  \frac{1}{2}\sum_{t = 1}^T\log \det(P^+_{t|t-1}) - \log \det(P_{t|t})\nn\\
&\ =  \frac{1}{2}\log \det(P^+_{1|0})  - \frac{1}{2}\log \det(P_{T|T}) \nn\\
&\ \ \ + \frac{1}{2}\sum_{t = 1}^{T-1}\log \det(P^+_{t+1|t}) - \log \det(P_{t|t}).
\end{align}
Each term in the sum can be written as
\begin{align}
&\ \log \det(P^+_{t+1|t}) - \log \det(P_{t|t})\nn\\
&\ \stackrel{(a)}=  - \log \det(I+P_{t+1|t}\snr^Y_t) + \log \det P_{t+1|t}P_{t|t}^{-1} \nn\\
&\ \stackrel{(b)}=  - \log \det(I+P_{t+1|t}\snr^Y_t) + \log \det W_t \nn\\
&\ \ + \log \det (P_{t|t}^{-1} + A_t^{\tr}W_t^{-1}A_t) \nn\\
&\ \stackrel{(c)}=  - \log \det(I+P_{t+1|t}\snr^Y_t) + \log \det W_t \nn\\ &\ + \inf_{\Pi_t}  \log \det \Pi^{-1}_t \nn\\
& \ \ \ \ \ \text{s.t. }  0\prec \Pi_t \preceq (P_{t|t}^{-1} + A_t^{\tr}W_t^{-1}A_t)^{-1}\nn\\
& \stackrel{(d)}=  \inf_{\Pi_t} - \log \det(I+P_{t+1|t}\snr^Y_t) - \log \det \Pi_t + \log \det W_t\nn\\
& \ \ \ \ \ \ \text{s.t. }  \begin{bmatrix} P_{t|t} - \Pi_t& P_{t|t}A_t^{\tr}\\A_tP_{t|t}& A_tP_{t|t}A_t^{\tr} +W_t
\end{bmatrix}\succeq0,  \Pi_t\succ0,
\end{align}
where $(a)$ follows from Lemma~\ref{lemma:covariances}, $(b)$ follows from Lemma~\ref{lemma:covariances} and Sylvester's determinant theorem, $(c)$ follows from introducing an auxiliary positive definite matrix $\Pi_t$ and from the monotonicity of $\log\det(\cdot)$ and, finally, $(d)$ follows from thematrix inversion lemma and Schur complement.

We will now convert the constraints to have a standard LMI form. First, note that the objective has no dependence on $\snr^F_t$. Thus, we can reduce this variable in the constraints of the optimization in \eqref{eq:proof_sdp_OP} as
\begin{align}
        (P^+_{t|t-1})^{-1} &= (P_{t|t-1})^{-1}+\snr^Y_t\nn\\
    P^{-1}_{t|t}&\succeq (P^{+}_{t|t-1})^{-1}\nn\\
    P_{t+1|t} &=  A_{t}P_{t|t}A_{t}^{\tr} + W_{t}.
\end{align}
The first two constraints can be combined as
\begin{align}
    P^{-1}_{t|t}&\succeq (P_{t|t-1})^{-1}+\snr^Y_t\nn\\
    & = (P_{t|t-1})^{-1} + C_t^{\tr}V_t^{-1}C_t.
\end{align}
By taking the inverse of both sides and applying the matrix inversion lemma, we can equivalently write the resulted inequality, using the Schur complement of a matrix, as
\begin{align}\label{eq:proof_SDP_LMI}
\Omega_t&=    \begin{bmatrix} P_{t|t-1} - P_{t|t}& P_{t|t-1}C_t^{\tr}\\C_tP_{t|t-1}& C_tP_{t|t-1}C_t^{\tr} +V_t
\end{bmatrix}\succeq0.
\end{align}
The derivation is completed by substituting in \eqref{eq:proof_SDP_LMI} $    P_{t+1|t} =  A_{t}P_{t|t}A_{t}^{\tr} + W_{t}$ for $t=1,\dots,T-1$.

To summarize, we showed that the optimization problem (up to the constant $\Lambda)$ is:
\begin{align}
        &\min_{\{P_{t|t}\}_{t=1}^T,\{\Pi_t\}_{t=1}^{T-1}}~ \Lambda -\frac{1}{2} \log \det(P_{T|T}) \nn\\
        &+  \frac{1}{2}\sum_{t=1}^{T-1}- \log \det(I+(\snr^Y_t)^{\frac1{2}}P_{t+1|t}(\snr^Y_t)^{\frac1{2}}) - \log \det \Pi_t \nn\\
        &\mspace{50mu}\text{s.t. }  \Trace(\Phi_1 P_{1|0}) + \sum_{t = 1}^T \Trace\left(\Theta_t P_t\right) + \Trace(S_tW_t) \leq d, \nn\\
        &\mspace{50mu}\begin{bmatrix} P_{t|t} - \Pi_t& P_{t|t}A_t^{\tr}\\A_tP_{t|t}& A_tP_{t|t}A_t^{\tr} +W_t
        \end{bmatrix}\succeq0,  \Pi_t\succ0,\nn\\
        &\mspace{50mu}\Omega_t\succeq0,
\end{align}
where $\Lambda = \frac{1}{2}\sum_{t = 1}^{T-1}\log \det W_t + \frac{1}{2}\log \det(P^+_{1|0})$. To obtain the closed form in Theorem \ref{th:SDP}, we substitute $(P_{1|0}^+)^{-1} = P_{1|0}^{-1} + \snr^Y_1$ and define $\Pi_T \triangleq P_{T|T}$.


\subsection{Proof of Theorem \ref{th:SDP_infinite}}\label{subsec:proof_infinite}
To simplify notation, we define the objective of the optimization problem as
\begin{align}
f_{s}(P)&\triangleq \frac1{2}\log\det W - \frac{1}{2} \log \det(I+(APA^T+W)\snr^Y) \nn\\
&\ \ + \frac{1}{2} \log \det (P^{-1} + A^{\tr} W^{-1}A),
\end{align}
and the constraints set as
\begin{align}\label{eq:proof_inf_D0}
\mathcal D_0 \triangleq \{P\succ 0| \Trace(\Theta P) + \Trace\left(W \bar{S}\right) \le D, \Omega(P)\succeq0 \ \},
\end{align}
where
\begin{align*}
\Omega(P)&\triangleq \begin{bmatrix} APA^{\tr} + W - P & (APA^{\tr} +W)C^{\tr}\\C(APA^{\tr} +W)& C(APA^{\tr} +W)C^{\tr} +V
\end{bmatrix}.
\end{align*}
We also define
\begin{align}\label{eq:proof_inf_RT}
    &R_T(P)\triangleq\frac1{T}\left.[\frac1{2}\log \det(I + (APA^{\tr}+W)\snr^Y) \right.\nn\\
&\  \ \left. - \frac1{2}\log \det(APA^{\tr}+W)\right.],
\end{align}
and present two technical lemmas needed for the proof of Theorem \ref{th:SDP_infinite}.
\begin{lemma}\label{lemma:inf_lim_exists}
For a sequence of matrices $\{P_{t|t}\}_{t\ge 1}$, let $\bar{P}_T\triangleq \frac{1}{T}\sum_{t=1}^T P_{t|t}$ be their uniform convex combination. Then, for any sequence $\{P_{t|t}\}_{t\ge1}$ that satisfies $\Omega_t\succeq 0$, there exists $\{T_i\}_{i\ge 1}$ such that $\lim_{i\to\infty}\bar{P}_{T_i} \in \mathcal D_0$.
\end{lemma}
\begin{lemma}\label{lemma:negative}
Let $T_i$ be a sequence as in Lemma \ref{lemma:inf_lim_exists}, and $\{P_{T_i|T_i}\}_{t\ge1}$ is a sequence that satisfies the constraints in \eqref{eq:th_main}. Then, under the conditions of Theorem \ref{th:SDP_infinite},
\begin{align*}
    \limsup_{i\to\infty} R_{T_i}(P_{T_i|T_i})\ge0.
\end{align*}
\end{lemma}
The proofs of Lemma \ref{lemma:inf_lim_exists} and Lemma \ref{lemma:negative} appear below. We are now ready to prove the main result in this section.
\begin{proof}[Proof of Theorem \ref{th:SDP_infinite}]
The optimal structure for the policy derived in Theorem \ref{th:structure} is true for any time horizon. Therefore, we can utilize Theorem \ref{th:structure} to write the optimization problem over the time-varying decision variables policy as
\begin{align}\label{eq:proof_infi_form1}
        &\inf_{\{P_{t|t}\succ0\}_{t\in\mathbb{N}}}~ \limsup_{T\to\infty}\frac1{T}(\Lambda_T - \frac1{2}\log \det P_{T|T} ) \nn\\
        &\ + \frac{1}{2T}\sum_{t=1}^{T-1} \log \det (P_{t|t}^{-1} + A^{\tr} W^{-1}A) \nn\\
        & \ - \frac{1}{2T}\sum_{t=1}^{T-1} \log \det(I+(AP_{t|t}A^{\tr} + W)\snr^Y)  \nn\\
        &\text{s.t. } \quad \limsup_{T\to\infty} \frac1{T}\Trace(\Phi_1 P_{1|0}) \nn\\
        & \ + \frac1{T}\sum_{t = 1}^T \Trace\left(\Theta_t P_{t|t}\right) + \Trace(S_tW) \leq \Gamma, \nn\\
        &\Omega_t=    \begin{bmatrix} P_{t|t-1} - P_{t|t}& P_{t|t-1}C_t^{\tr}\\C_tP_{t|t-1}& C_tP_{t|t-1}C_t^{\tr} +V_t\end{bmatrix}\succeq0, \ t\in \mathbb{N},
\end{align}
where the constant matrices $\Phi $ and $\Theta_t$ are given in Theorem \ref{th:SDP}, and the constant $\Lambda_T$ is given by
\begin{align}
    \Lambda_T &= - \frac1{2}\log \det( P_{1|0}^{-1}+\snr^Y) + \frac{1}{2} \sum_{t=1}^{T-1}\log \det W.
\end{align}
By taking the limit in \eqref{eq:proof_infi_form1} over the time-independent quantities, we have $\frac1{T}\Lambda_T\to \frac1{2}\log\det W$ and
$\frac1{T}\Trace(\Phi_1 P_{1|0})\to 0$. The LQG cost constraint is simplified by noting that $S_t$ converges to the stabilizing solution of the Riccati equation in \eqref{eq:S_bar_riccati}, $\bar{S}$, by the assumption that the $(A,B)$ is stabilizable and $(A,Q^{1/2})$ is controllable on the unit circle. This in turn implies that $\Theta_t\to\Theta$.

Next, we define
\begin{align}
&f_{r,T}(\{P_{t|t}\}_{t=1}^T)\triangleq \frac1{2T}\sum_{t=1}^T \log \det (P_{t|t}^{-1} + A^{\tr} W^{-1}A) \\
& + \frac1{2}\log\det W - \frac1{2T} \sum_{t=1}^T \log \det(I+(AP_{t|t}A^T+W)\snr^Y),\nn
\end{align}
in order to compactly express the optimization problem as
\begin{align}\label{eq:proof_infi_form2}
        &\min_{\{P_{t|t}\}_{t\in\mathbb{N}}}~ \limsup_{T\to\infty} R_T(P_{T|T}) + f_{r,T}(\{P_{t|t}\}_{t=1}^T)\nn\\
        &\text{s.t. } \quad \limsup_{T\to\infty} \frac1{T}\sum_{t = 1}^T \Trace\left(\Theta P_{t|t}\right) + \Trace(\bar{S}W) \leq \Gamma, \nn\\
        &\Omega_t=    \begin{bmatrix} P_{t|t-1} - P_{t|t}& P_{t|t-1}C_t^{\tr}\\C_tP_{t|t-1}& C_tP_{t|t-1}C_t^{\tr} +V_t\end{bmatrix}\succeq0  \nn \\
        &P_{t|t-1}= A P_{t-1|t-1}A^T+W \ \ \forall t\in \mathbb{N},
\end{align}
where $R_T(\cdot)$ was defined in \eqref{eq:proof_inf_RT}.

We can now present the main steps that constitute the proof of the lower bound in Theorem \ref{th:SDP_infinite}:
\begin{align}\label{eq:proof_inf_derivation}
   &\min\limsup_{T\to\infty} R_{T} + f_{r,T}(\{P_{t|t}\}_{t=1}^{T})\nn\\
   &\ \ \stackrel{(a)}\ge \min\limsup_{i\to\infty} R_{T_i} + f_{r,T_i}(\{P_{t|t}\}_{t=1}^{T_i}) \nn\\
   &\ \ \stackrel{(b)}\ge \min\limsup_{i\to\infty} f_{r,T_i}(\{P_{t|t}\}_{t=1}^{T_i})\nn\\
   &\ \ \stackrel{(c)}\ge \min\limsup_{i\to\infty} f_{s}(\bar{P}_{T_i})\nn\\
   &\ \ \stackrel{(d)}\ge \min_{P\in\mathcal D_0} f_{s}(P),
\end{align}
where:
\begin{itemize}
    \item[(a)] follows relaxing to the limit supremum from $T_i$ to $T$;
    \item[(b)] follows from the non-negativity of $\limsup R_{T_i}$. shown below as Lemma \ref{lemma:negative};
    \item[(c)] follows from the convexity of the function $- \log \det(I+(APA^T+W)\snr^Y)-\log\det(P^{-1} + A^{\tr} W^{-1}A)$;
\item[(d)] follows from Lemma \ref{lemma:inf_lim_exists}.
\end{itemize}
Note that the left-hand side of \eqref{eq:proof_inf_derivation} is the optimization problem derived in \eqref{eq:proof_infi_form2}. Thus, we showed a single-letter lower bound to the optimization problem \eqref{eq:Optimization_def_inf}. Furthermore, for a matrix $P\in\mathcal D_0$, the time-invariant cartesian product $\bigotimes_{t\ge1}P$ satisfies the time-dependent constraints in \eqref{eq:proof_infi_form2} and the resulting objective in the optimization problem \eqref{eq:proof_infi_form2} is $f_s(P)$.
These steps conclude that \eqref{eq:infinite_optimization} is a lower to \eqref{eq:Optimization_def_inf}.

In the last step, we show that there exists a time-invariant policy that achieves the lower bound in \eqref{eq:proof_infi_form2}. Namely, we show that for each covariance matrix $P\in\mathcal D_0$ in \eqref{eq:proof_inf_derivation}, there exists a time-invariant policy that makes the objective in \eqref{eq:proof_infi_form2} equal to $f_s(P)$. To this end, we  construct a time-invariant policy given by the pair $(D,M)$ as follows: given $P$, compute $(D,M)$ matrices using the SVD decomposition  $D^{\tr}M^{-1}D = P^{-1} - (APA^{\tr}+W)^{-1} - \snr^{Y}$ with $M\succ0$. By construction, we have that $P$ is a solution to the Ricatti equation
\begin{align}\label{eq:inf_proof_Riccati}
    APA^{\tr} - P + W - APH^{\tr}(HPH^{\tr}+\tilde{V})^{-1}HPA^{\tr} &= 0,
\end{align}
where $H\triangleq\begin{bmatrix}
C \\
D
\end{bmatrix} , \tilde{V}\triangleq\begin{bmatrix}
V &0\\
0&M
\end{bmatrix}$. The closed-loop Riccati equation of \eqref{eq:inf_proof_Riccati} can be re-written as
\begin{align}
  (A-\Xi H)P(A-\Xi H)^{\tr} - P + W + \Xi \tilde{V}\Xi^{\tr}= 0,
\end{align}
where $\Xi= APH^{\tr}(HPH^{\tr}+\tilde{V})^{-1}$. From the assumption $W\succ 0$, the closed-loop system $(A-\Xi H)$ is stable. Therefore, $(A,H)$ is detectable and $P$ is the unique maximal solution to the Riccati equation. The detectability of $(A,H)$ and $P_{1|0}\succeq P$ \footnote{If this is not the case, one can increase the covariance by ignoring the measurements for several time instances} guarantee the convergence of the forward Riccati recursion
\begin{align}\label{eq:proof_inf_matrices_time}
        P_{t|t}&= AP_{t-1|t-1}A^{\tr} + W \nn\\
        &\ - AP_{t-1|t-1}H^{\tr}(HP_{t-1|t-1}H^{\tr}+\tilde{V})^{-1}HP_{t-1|t-1}A^{\tr}
\end{align}
to the maximal solution of the Riccati equation. The proof is completed by computing the limit in \eqref{eq:proof_infi_form2} using the convergence of \eqref{eq:proof_inf_matrices_time}. This concludes the proof of Theorem \ref{th:SDP_infinite}.
\end{proof}

\begin{proof}[Proof of Lemma \ref{lemma:inf_lim_exists}]
For $\epsilon>0$, define $\mathcal D_\epsilon \triangleq \{P\succ 0| \Trace(\Theta P) + \Trace(W\bar{S})\le \Gamma \ \text{and} \ \eqref{eq:inf_stationary_matri}\}$.
\begin{align}\label{eq:inf_stationary_matri}
&\begin{bmatrix} APA^{\tr} + W - P & (APA^{\tr} +W)C^{\tr}\\C(APA^{\tr} +W)& C(APA^{\tr} +W)C^{\tr} +V
\end{bmatrix}\succeq \begin{bmatrix} -\epsilon I&0\\0&0\end{bmatrix}
\end{align}
By the assumption, $\Omega_t \succeq0,$ for all $t\ge1$. Consider the time-invariant constraint
\begin{align}
 \Omega(\bar{P}_T)
&= \frac1{T} \sum_{t=1}^{T+1} \Omega_t \nn\\
&\ \ + \begin{bmatrix}  - \frac1{T} P_{1|0} + \frac1{T} P_{T+1|T+1} & - \frac1{T} P_{1|0}C^{\tr}\\ - \frac1{T} CP_{1|0}& - \frac1{T} CP_{1|0}C^{\tr} -\frac1{T} V
\end{bmatrix}\nn\\
&\succeq \begin{bmatrix}  - \frac1{T} P_{1|0} & - \frac1{T} P_{1|0}C^{\tr}\\ - \frac1{T} CP_{1|0}& - \frac1{T} CP_{1|0}C^{\tr} -\frac1{T} V
\end{bmatrix}.
\end{align}
By the boundedness of $P_{1|0}$ and $V$, for each $\epsilon>0$, there exists $T_\epsilon$ such that $\bar{P}_{T}\in\mathcal D_\epsilon$ for all $T\ge T_\epsilon$. The LQG constraint is trivially satisfied at $\bar{P}_T$ by the linearity of the trace operator. The boundedness of $\mathcal D_\epsilon$ and $\mathcal D_0$ follows from the boundeness of a larger set in \cite{tanaka_inf}. Therefore, these sets are compact, and there exists a limit point in $\cap_{\epsilon>0} \mathcal D_\epsilon$ equal to $\mathcal D_0$.
\end{proof}
\begin{proof}[Proof of Lemma \ref{lemma:negative}]
We can write
\begin{align}
R_{T_i}(P_{T_i|T_i}) &=
    \frac1{T_i} \left[\frac1{2}\log \det(I + (AP_{T_i|T_i}A^{\tr}+W)\snr_Y) \right. \nn\\
    &\ \left. - \frac1{2}\log \det(AP_{T_i|T_i}A^{\tr}+W)\right] \nn\\
    &= \frac1{T_i} \frac1{2} \log\det (P^{-1}_{T_i|T_i-1} + \snr^Y)^{-1}.
\end{align}
In order to show the non-negativity of the limit, one can follow the steps in \cite[Lemma $3$]{tanaka_inf} along with the fact that the limiting error covariance $\lim_{i\to\infty}\bar{P}_{T_i}$, must be bounded in all directions that are not orthogonal to $A$.
\end{proof}

\subsection{Proof of Corollary \ref{co:scalarSDP}}\label{sec:scalarSDP}
\begin{proof}
When $A, B, W, C, V$ are scalars and $Q = R = 1$, $\bar{S}, K, \Theta$ also become scalars and the optimization~\eqref{eq:infinite_optimization} reduces to
\begin{align}\label{eq:scalar_infinite_optimization}
        &\inf_{P,\Pi\in\mathbb{R}}~ \frac1{2} \log W - \frac{1}{2} \log (1+\snr^Y(A^2P+W)) \nn\\&\quad\quad\quad\quad - \frac{1}{2} \log \Pi \nn\\
        \text{s.t. } &\quad \Theta P + W \bar{S}\leq \Gamma, \nn\\
        &\quad\begin{bmatrix} P - \Pi& AP\\AP & A^2P +W
        &\end{bmatrix}\succeq0,\quad\Pi > 0\nn\\
        &\quad \begin{bmatrix} A^2P + W - P& C(A^2P + W)\\C(A^2P + W)& C^2(A^2P + W) +V
\end{bmatrix}\succeq 0.
\end{align}
Since $|A|>1$, we can easily verify using~\eqref{eq:scalarSbar} that $\bar{S} > 0$ (in fact, $\bar{S} > 1$) and thus $K > 0$ and $\Theta > 0$. To simplify the PSD constraints, notice that a $2\times 2$ matrix $X$ is positive semidefinite if and only if $\Trace(X)\geq 0$ and $\det(X)\geq 0$. With this observation, we can simplify the constraints further as
\begin{align}\label{eq:scalar_P}
        & 0 < P\leq \min\left\{\frac{\Gamma-W \bar{S}}{\Theta},~P^\star\right\}, \nn\\
        & 0 < \Pi \leq \min\left\{ P + A^2P +W, ~\frac{WP}{A^2P +W}\right\},
\end{align}
where $P^\star$ is defined as the unique positive solution to~\eqref{eq:scalarPstar}. Since
\begin{align}
    \frac{WP}{A^2P +W} \leq \frac{W}{A^2} \leq W < P + A^2P +W,
\end{align}
the constraint on $\Pi$ is further simplified to
$0< \Pi\leq \frac{WP}{A^2P +W}$.
Also notice that for any fixed feasible $P$, $\log \Pi$ is maximized at $\frac{WP}{A^2P +W}$. Then, the optimization problem is further simplified as
\begin{align}\label{eq:scalar_simplified_infinite_optimization}
        \inf_{P\in\mathbb{R}:~\eqref{eq:scalar_P}}~ \frac1{2} \log \left(A^2  + \frac{W}{P} \right) - \frac{1}{2} \log\left(1+\frac{C^2}{V}(A^2P+W) \right).
\end{align}
Let $g(P)$ be the function such that the objective function is written as $-\frac{1}{2}\log g(P)$, that is,
\begin{align}
    g(P) \triangleq \frac{P(1 + \frac{C^2}{V}(A^2P + W))}{A^2P +W}.
\end{align}
It then suffices to show that $g(P)$ is an increasing function in $P>0$. This can be accomplished by rewriting $g(P)$ as
\begin{align}
    g(P) = \frac{1}{A^2}\left(\frac{C^2}{V}(A^2P + W)- \frac{W}{A^2P + W} -\frac{C^2W}{V}+1 \right),
\end{align}
which is increasing in $P>0$ since $W>0$. Therefore, when $W\bar{S} < \Gamma \leq W\bar{S} + \Theta P^\star$, the optimal value is given by
\begin{align}
    -\frac{1}{2}\log g\left(\frac{\Gamma - W\bar{S}}{\Theta}\right),
\end{align}
which equals~\eqref{eq:scalar_sol1}. When $\Gamma > W\bar{S} + \Theta P^\star$, the optimal value is given by
\begin{align}
    -\frac{1}{2}\log g\left(P^\star\right).
\end{align}
Finally, one can verify that $g(P^\star) = 1$, which implies that the optimal value equals 0 for $\Gamma > W\bar{S} + \Theta P^\star$ using the fact that $P^\star$ is the solution to~\eqref{eq:scalarPstar}.
\end{proof}

\section{Conclusions}\label{sec:conclusions}
In this paper, we formulated and solved an optimization problem for the LQG setting with an additional communication link. We first showed that the optimal encoding realization is a memoryless Gaussian measurement of the state and the optimal control is the standard LQG control law. We then utilized the policy structure to show the main result that the minimization of the conditional directed information subject to a control constraint can be formulated as a standard convex optimization problem. For the finite-horizon regime, that convex optimization problem consists of a sequence of decision variables, while in the infinite-horizon regime, it simplifies to a single-letter optimization problem. The examples illustrate the benefits of the LQG setting with side information compared to the setting without side information even if the measurement has a low SNR.

\bibliography{Arxiv.bbl}

\begin{thebibliography}{10}
\providecommand{\url}[1]{#1}
\csname url@samestyle\endcsname
\providecommand{\newblock}{\relax}
\providecommand{\bibinfo}[2]{#2}
\providecommand{\BIBentrySTDinterwordspacing}{\spaceskip=0pt\relax}
\providecommand{\BIBentryALTinterwordstretchfactor}{4}
\providecommand{\BIBentryALTinterwordspacing}{\spaceskip=\fontdimen2\font plus
\BIBentryALTinterwordstretchfactor\fontdimen3\font minus
  \fontdimen4\font\relax}
\providecommand{\BIBforeignlanguage}[2]{{%
\expandafter\ifx\csname l@#1\endcsname\relax
\typeout{** WARNING: IEEEtran.bst: No hyphenation pattern has been}%
\typeout{** loaded for the language `#1'. Using the pattern for}%
\typeout{** the default language instead.}%
\else
\language=\csname l@#1\endcsname
\fi
#2}}
\providecommand{\BIBdecl}{\relax}
\BIBdecl

\bibitem{sabag_cdc}
O.~Sabag, P.~Tian, V.~Kostina, and B.~Hassibi, ``The minimal directed
  information needed to improve the {LQG} cost,'' in \emph{2020 59th IEEE
  Conference on Decision and Control (CDC)}, Dec. 2020, pp. 1842--1847.

\bibitem{ConCom_Elia}
N.~{Elia} and S.~K. {Mitter}, ``Stabilization of linear systems with limited
  information,'' \emph{IEEE Trans. Autom. Control}, vol.~46, no.~9, pp.
  1384--1400, Sep. 2001.

\bibitem{ConCom_matveev}
A.~S. Matveev and A.~V. Savkin, ``An analogue of {S}hannon information theory
  for detection and stabilization via noisy discrete communication channels,''
  \emph{SIAM Journal on Control and Optimization}, vol.~46, no.~4, pp.
  1323--1367, Sep. 2007.

\bibitem{ConCom_matveev2}
A.~S. Matveev, ``State estimation via limited capacity noisy communication
  channels,'' \emph{Math. Control Signals Syst}, vol.~20, pp. 1--35, Mar. 2008.

\bibitem{ConCom_Liberzon2}
D.~{Liberzon}, ``On stabilization of linear systems with limited information,''
  \emph{IEEE Trans. Autom. Control}, vol.~48, no.~2, pp. 304--307, Feb. 2003.

\bibitem{YukselStability}
S.~{Yuksel}, ``Stochastic stabilization of noisy linear systems with fixed-rate
  limited feedback,'' \emph{IEEE Trans. Autom. Control}, vol.~55, no.~12, pp.
  2847--2853, Dec. 2010.

\bibitem{TaikondaSahaiMitter}
S.~{Tatikonda}, A.~{Sahai}, and S.~{Mitter}, ``Stochastic linear control over a
  communication channel,'' \emph{IEEE Transactions on Automatic Control},
  vol.~49, no.~9, pp. 1549--1561, Sep. 2004.

\bibitem{KostinaBitsArxiv}
V.~{Kostina}, Y.~{Peres}, G.~{Ranade}, and M.~{Sellke}, ``Exact minimum number
  of bits to stabilize a linear system,'' in \emph{2018 IEEE Conference on
  Decision and Control (CDC)}, Dec. 2018, pp. 453--458.

\bibitem{sabag_isit_fixedrate}
O.~Sabag, V.~Kostina, and B.~Hassibi, ``Stabilizing dynamical systems with
  fixed-rate feedback using constrained quantizers,'' in \emph{2020 IEEE Int.
  Symp. Inf. Theory (ISIT)}, Jun. 2020, pp. 2855--2860.

\bibitem{silva_framework}
E.~I. Silva, M.~S. Derpich, and J.~Ostergaard, ``A framework for control system
  design subject to average data-rate constraints,'' \emph{IEEE Trans. Autom.
  Control}, vol.~56, no.~8, pp. 1886--1899, 2011.

\bibitem{Milan_UB_causalRDF_concom}
M.~S. Derpich and J.~Ostergaard, ``Improved upper bounds to the causal
  quadratic rate-distortion function for gaussian stationary sources,''
  \emph{IEEE Transactions on Information Theory}, vol.~58, no.~5, pp.
  3131--3152, May 2012.

\bibitem{victoria_hassibi_tradeoff}
V.~{Kostina} and B.~{Hassibi}, ``Rate-cost tradeoffs in control,'' \emph{IEEE
  Trans. Autom. Control}, vol.~64, no.~11, pp. 4525--4540, Nov. 2019.

\bibitem{anatoly_fixedrate}
\BIBentryALTinterwordspacing
A.~Khina, Y.~Nakahira, Y.~Su, H.~Yildiz, and B.~Hassibi, ``Algorithms for
  optimal control with fixed-rate feedback,'' 2018. [Online]. Available:
  \url{http://arxiv.org/abs/1809.04917}
\BIBentrySTDinterwordspacing

\bibitem{charalambos_nonanticipative}
C.~D. {Charalambous}, P.~A. {Stavrou}, and N.~U. {Ahmed}, ``Nonanticipative
  rate distortion function and relations to filtering theory,'' \emph{IEEE
  Trans. Autom. Control}, vol.~59, no.~4, pp. 937--952, Apr. 2014.

\bibitem{STAVROU_SI}
P.~A. Stavrou and M.~Skoglund, ``{LQG} control and linear policies for noisy
  communication links with synchronized side information at the decoder,''
  \emph{Automatica}, vol. 123, p. 109306, Oct. 2021.

\bibitem{tishby_LQG}
R.~{Fox} and N.~{Tishby}, ``Minimum-information {LQG} control part {I}:
  Memoryless controllers,'' in \emph{2016 IEEE 55th Conference on Decision and
  Control (CDC)}, Dec. 2016, pp. 5610--5616.

\bibitem{Tanaka_SDP}
T.~{Tanaka}, P.~M. {Esfahani}, and S.~K. {Mitter}, ``{LQG} control with minimum
  directed information: Semidefinite programming approach,'' \emph{IEEE Trans.
  Autom. Control}, vol.~63, no.~1, pp. 37--52, Jan 2018.

\bibitem{Kostina_SI_Allerton}
V.~{Kostina} and B.~{Hassibi}, ``Rate-cost tradeoffs in scalar {LQG} control
  and tracking with side information,'' in \emph{2018 56th Annual Allerton
  Conference on Communication, Control, and Computing (Allerton)}, Oct 2018,
  pp. 421--428.

\bibitem{Tanaka_SI}
T.~C. Cuvelier and T.~Tanaka, ``Rate of prefix-free codes in {LQG} control
  systems with side information,'' in \emph{2021 55th Annual Conference on
  Information Sciences and Systems (CISS)}, Mar. 2021, pp. 1--6.

\bibitem{lev2020schemes}
O.~Lev and A.~Khina, ``Schemes for {LQG} control over {G}aussian channels with
  side information,'' available at arxiv.org/abs/2004.03927.

\bibitem{PermuterWeissmanGoldsmith09}
H.~H. Permuter, T.~Weissman, and A.~J. Goldsmith, ``Finite state channels with
  time-invariant deterministic feedback,'' \emph{IEEE Trans. Inf. Theory},
  vol.~55, no.~2, pp. 644--662, Feb. 2009.

\bibitem{TatikondaMitter_IT09}
S.~Tatikonda and S.~Mitter, ``The capacity of channels with feedback,''
  \emph{IEEE Trans. Inf. Theory}, vol.~55, no.~1, pp. 323--349, Jan. 2009.

\bibitem{Kramer98}
G.~Kramer, ``Directed information for channels with feedback,'' Ph.{D}.
  Dissertation, Swiss Federal Institute of Technology (ETH) Zurich, 1998.

\bibitem{tanaka_scheme}
T.~{Tanaka}, K.~H. {Johansson}, T.~{Oechtering}, H.~{Sandberg}, and
  M.~{Skoglund}, ``Rate of prefix-free codes in {LQG} control systems,'' in
  \emph{2016 IEEE International Symposium on Information Theory (ISIT)}, Aug.
  2016, pp. 2399--2403.

\bibitem{Tanaka_Gaussian_AC}
T.~{Tanaka}, K.~K. {Kim}, P.~A. {Parrilo}, and S.~K. {Mitter}, ``Semidefinite
  programming approach to {G}aussian sequential rate-distortion trade-offs,''
  \emph{IEEE Trans. Autom. Control}, vol.~62, no.~4, pp. 1896--1910, Aug. 2017.

\bibitem{lev_khina_LB}
O.~Lev and A.~Khina, ``{G}auss–{M}arkov source tracking with side
  information: Lower bounds,'' in \emph{2020 International Symposium on
  Information Theory and Its Applications (ISITA)}, Oct. 2020, pp. 66--70.

\bibitem{Tuncel_SI}
X.~Chen and E.~Tuncel, ``Zero-delay joint source-channel coding using hybrid
  digital-analog schemes in the {W}yner-{Z}iv setting,'' \emph{IEEE
  Transactions on Communications}, vol.~62, no.~2, pp. 726--735, Feb. 2014.

\bibitem{kochman_modulo_SI}
Y.~Kochman and R.~Zamir, ``Joint {W}yner–{Z}iv/dirty-paper coding by
  modulo-lattice modulation,'' \emph{IEEE Transactions on Information Theory},
  vol.~55, no.~11, pp. 4878--4889, Oct. 2009.

\bibitem{Massey90}
J.~Massey, ``Causality, feedback and directed information,'' \emph{Proc. Int.
  Symp. Inf. Theory Applic. (ISITA-90)}, pp. 303--305, Nov. 1990.

\bibitem{boyd_detmax}
L.~Vandenberghe, S.~Boyd, and S.-P. Wu, ``Determinant maximization with linear
  matrix inequality constraints,'' \emph{SIAM J. Matrix Anal. Appl.}, vol.~19,
  no.~2, p. 499–533, Apr. 1998.

\bibitem{cvx}
M.~Grant and S.~Boyd, ``{CVX}: Matlab software for disciplined convex
  programming, version 2.1,'' \url{http://cvxr.com/cvx}, Mar. 2014.

\bibitem{yalmip}
J.~L{\"{o}}fberg, ``Yalmip : A toolbox for modeling and optimization in
  matlab,'' in \emph{In Proceedings of the CACSD Conference}, Taipei, Taiwan,
  2004.

\bibitem{NairEvans04}
G.~N. Nair and R.~J. Evans, ``Stabilizability of stochastic linear systems with
  finite feedback data rates,'' \emph{SIAM Journal on Control and
  Optimization}, vol.~43, no.~2, pp. 413--436, Jul. 2004.

\bibitem{ElGamal}
A.~El~Gamal and Y.-H. Kim., \emph{Network Information Theory}.\hskip 1em plus
  0.5em minus 0.4em\relax Cambridge University Press, 2011.

\bibitem{tanaka_inf}
T.~{Tanaka}, ``Semidefinite representation of sequential rate-distortion
  function for stationary {G}auss-{M}arkov processes,'' in \emph{2015 IEEE
  Conference on Control Applications (CCA)}, Sep. 2015, pp. 1217--1222.

\end{thebibliography}
\bibliographystyle{IEEEtran}

\end{document}